\documentclass[12pt]{amsart}

\newcommand{\nc}{\newcommand}
\nc{\G}{{\Gamma}} \nc{\BC}{{\mathbb C}} \nc{\BQ}{{\mathbb Q}}
\nc{\BR}{{\mathbb R}} \nc{\BZ}{{\mathbb Z}} \nc{\BP}{{\mathbb P}}
\nc{\BN}{{\mathbb N}} \nc{\BM}{{\mathbb M}}
\nc{\fH}{{\mathbb H}}

\nc{\PS}{{\mbox{PSL}_2(\BZ)}} \nc{\SL}{{\mbox{SL}_2(\BZ)}}
\nc{\SR}{{\mbox{SL}_2(\BR)}} \nc{\PR}{{\mbox{PSL}_2(\BR)}}
\nc{\GL}{{\mbox{GL}_2^+(\BQ)}} \nc{\PQ}{{\mbox{PGL}_2^+(\BQ)}}
\nc{\GR}{{\mbox{GL}_2^+(\BR)}} \nc{\PG}{{\mbox{PGL}_2^+(\BR)}}
\nc{\GC}{{\mbox{GL}_2(\BC)}} \nc{\SC}{{\mbox{SL}_2(\BC)}}
\nc{\f}{{\mathcal{F}(\fH)}}
\nc{\Cc}{\widehat{\BC}}
\nc{\e}{{E_{\rho}(\G)}}
\nc{\g}{{\gamma}}
\nc{\vm}{{V_{\rho}(\G)}}

\newtheorem{numbered}{}[section]
\newtheorem{thm}[numbered]{Theorem}

\newtheorem{remark}[numbered]{Remark}
\newtheorem{prop}[numbered]{Proposition}
\newtheorem{cor}[numbered]{Corollary}

\numberwithin{equation}{section}

\newcommand{\thmref}[1]{Theorem~\ref{#1}}
\newcommand{\propref}[1]{Proposition~\ref{#1}}

\newcommand{\corref}[1]{Corollary~\ref{#1}}

\begin{document}

\title{Equivariant functions and vector-valued modular forms}
\author[]{Hicham Saber}\author[]{Abdellah Sebbar}
\address{Department of Mathematics and Statistics, University of Ottawa, Ottawa Ontario K1N 6N5 Canada}
\email{hsabe083@uottawa.ca}
\email{asebbar@uottawa.ca}
\begin{abstract}
For any discrete group $\Gamma$ and any 2-dimensional complex representation $\rho$ of $\Gamma$, we introduce the notion  of
$\rho-$equivariant functions, and we show that they are parameterized by vector-valued modular forms. We also provide examples arising from
the monodromy of differential equations.
\end{abstract}
\maketitle
\section{Introduction}
Throughout this paper, by a discrete group, we mean
 a finitely generated Fuchsian group of the first kind, acting on the upper half-plane
$\fH=\{z\in\BC :\, \mbox{Im}\,z>0\}$. Let $\Gamma$ be such a group.
Let $\rho:\G\longrightarrow \GC$ be a 2-dimensional complex representation
of $\G$.  A meromorphic function $h$ on $\fH$ is called a \textit{$\rho-$equivariant function} with respect to $\G$ if
\begin{equation}\label{equiv1}
h(\gamma\cdot z)\,=\,\rho(\gamma)\cdot h(z)\,\ \mbox{ for all }\, z\in\fH\,,\ \gamma\in\G\,,
\end{equation}
where the action on both sides is by linear fractional transformations.
The set of $\rho-$equivariant functions for $\G$ will be denoted by $\e$.

In the case $\rho$ is the defining representation of $\G$, that is $\rho(\gamma)=\gamma$ for all $\gamma\in\G$, then elements of $\e$ are simply
called equivariant functions. These were studied extensively in \cite{2,3,4} and have various connections to modular forms, quasi-modular forms,
elliptic functions and to sections of the canonical line bundle of $X(\G)=\overline{\G\backslash\fH}$. In particular, one shows that the set of equivariant
functions for a discrete group $\G$ without the trivial one $h_0(z)=z$ has a vector space structure isomorphic to the space of
weight 2 automorphic forms for $\G$.

In this paper, we will treat the general case where $\rho$ is an arbitrary 2-dimensional complex representation of $\G$.
The main result of this paper states that every $\rho-$equivariant function is parameterized by a 2-dimensional vector-valued modular form for $\rho$.
More precisely, if $F=(f_1,f_2)^t$ is a vector-valued modular form for $\G$ and $\rho$ (see Section 2), then $h_F=f_1/f_2$ is a $\rho-$equivariant function for $\G$. We will show that, in fact,
 every $\rho-$equivariant function arises in this way. To achieve this parametrization, we use the fact that
the Schwarz derivative of a $\rho-$equivariant function is a weight 4 automorphic form for $\G$, in addition to the knowledge of the existence of global solutions
to a certain second degree differential equation.

Finally, in the last section, we construct examples of $\rho-$equivariant functions when $\rho$ is the monodromy representation of second degree
ordinary differential equations.

\section{Vector-valued modular forms}
The theory of vector-valued modular forms  was introduced  long ago as a
higher dimensional generalization of the classical (scalar) modular forms.
 Let $\Gamma$ be a discrete group and let $\rho: \G\longrightarrow \mbox{GL}_n(\BC)$ be
an $n-$dimensional complex representation of $\G$. A vector-valued modular form
of integral weight $k$ for $\G$ and representation $\rho$ is an $n-$tuple $F(z)=(f_1(z),\ldots,f_n(z))^t$ of meromorphic functions on the complex upper half-plane $\fH$ satisfying
\begin{equation}\label{inv1}
F(z)|_k \gamma\,(z)\,=\,\rho(\gamma)\,F(z)\, \mbox{ for all } \gamma\in\G\,,\ z\in\fH\,,
\end{equation}
and some growth conditions at the cusps that are similar to those for   classical automorphic forms. The slash operator in \eqref{inv1} is defined as usual by $F|_k\gamma(z)=(cz+d)^{-k}F(\gamma\cdot z)$ where $\displaystyle\gamma=\binom{a\ b}{c\ d}$. The set of these vector-valued modular forms will be denoted by $\vm_k$, and we will refer to them as $\rho-$VMF of weight $k$.

We omit the multiplier system since it will not have any effect in this paper. If $F(z)$ is holomorphic, then it will be called a holomorphic vector-valued modular form.

The theory of vector-valued modular forms is fairly well understood when $\Gamma$ is the modular group. See for instance \cite{kn-ma1} and the references therein. However, except for the genus 0 subgroups of the modular group, it is not even clear that nonzero vector-valued modular forms exist.
In connection with equivariant functions, we have
the following   straightforward proposition.
\begin{prop}\label{direct}
Let $\G$ be a discrete group, and $\rho$ an arbitrary 2-dimensional complex representation of $\G$. If $F(z)=(f_1(z),f_2(z))^t$ is a
vector-valued modular form for $\rho$ of a certain weight, then $f_1(z)/f_2(z)$ is a $\rho-$equivariant function for $\G$.
\end{prop}
We will prove below that every $\rho-$equivariant function arises in this way.

\section{Differential equations}
Let $D$ be a domain in $\mathbb C$ and let $f$ be a  meromorphic function on  $D$. Its Schwarz derivative, $S(f)$, is defined by
$$
S(f)\,=\, \left(\frac{f''}{f'}\right)'-\frac{1}{2}\left(\frac{f''}{f'}\right)^2\,.
$$
This is an important tool in projective geometry and differential equations. The main properties that will be useful to us are summarized as follows (
see \cite{mckay-sebbar} for more details):

\begin{prop}\label{schwarz} We have
\begin{enumerate}
\item If $y_1$ and $y_2$ are two linearly independent solutions
to a differential equation $y''+Qy=0$  where $Q$ is a meromorphic function on $D$, then $S(y_1/y_2)=2Q$.
\item If $f$ and $g$ are two meromorphic functions on $D$, then  $S(f)=S(g)$ if and only if
$\displaystyle f=\frac{ag+b}{cg+d}$ for some $\displaystyle\binom{a\ b}{c\ d}\in\GC$.
\item $S(f\circ \gamma)(z)=(cz+d)^4S(f)$  provided $\gamma\cdot z\in D$, where $\displaystyle \gamma=\binom{*\ *}{c\ d}$.
\end{enumerate}
\end{prop}
In particular, we have
\begin{prop}\label{s-autom}
If $f$ is a $\rho$-equivariant for a discrete group $\G$, then $S(f)$ is an automorphic form of weight 4 for $\G$.
\end{prop}
Now, consider the second order ordinary differential equation (ODE)
$$x''+Px'+Qx=0,
$$
where $P$ and $Q$ are holomorphic functions in $D$. This ODE has two linearly independent holomorphic solutions in $D$ if $D$ is simply connected.
For a fixed $z_0\in D$, set
$$y(z)\,=\,x(z)\exp\left(\int_{z_0}^{z}\frac{1}{2}P(w)dw\right).$$
 The above ODE reduces to an ODE in normal form
 \begin{equation}\label{normal-form}
 y''\ +\ gy\ =\ 0\, ,
 \end{equation}
 with
 $$
  g=Q-\frac{1}{2}P'-\frac{1}{4}P^2.
 $$


 When the domain $D$ is not simply connected, we may not expect to
  find global solutions  to \eqref{normal-form} on $D$. However, under some conditions on $g$,
 global solutions do exist as it is illustrated in the following theorem which will be crucial for the rest of this paper.
 \begin{thm}\label{thm3.3}
 Let $D$ be a domain in $\BC$. Suppose $h$ is a nonconstant meromorphic function on $D$ such that $S(h)$ is holomorphic in $D$, and let $g=\frac{1}{2}S(h)$.
 Then the differential equation
 $$y''\,+\,gy\,=\,0$$
 has two linearly independent holomorphic solutions in $D$.
 \end{thm}
\begin{proof}
  Let $\{U_i, i\in I\}$ be a covering of $D$ by open discs with $\mbox{dim}\,V(U_i)= 2$ for all $i\in I$ where  $V(U_i)$ denotes the space of holomorphic solutions to $y''+gy\,=\,0$ on $U_i$. Choose $L_i$ and $K_i$ to form a basis for $V(U_i)$. Using property (1) of \propref{schwarz}, we have $S(K_i/L_i)=2g=S(h)$ on $U_i$. Now, using property (2) of \propref{schwarz}, we have $K_i/L_i=\alpha_i\cdot h$ for $\alpha_i\in\GC$.
  In the meantime, on each connected component $W$ of $U_i\cap U_j$,  we have
  $$
  (K_i,L_i)^t=\alpha_W(K_j,L_j)^t\ ,\ \ \alpha_W\in\GC\,,
  $$
since each of $(K_i,L_i)$ and $(K_j,L_j)$ is a basis of $V(W)$. Hence, on $W$ we have
$$
\frac{K_i}{L_i}\,=\,\alpha_W\cdot \frac{K_j}{L_j} \,,
$$
and therefore
$$
\alpha_ih\,=\,\alpha_W\alpha_jh\,.
$$
It follows that
$$
\alpha_i\alpha_j^{-1}\,=\,\alpha_W
$$
as $h$ is meromorphic and nonconstant and thus it takes more than three distinct values on the domain $D$.
Therefore, $\alpha_W$ does not depend on $W$. Moreover,  on $U_i\cap U_j$ we have
\begin{equation}\label{patch}
\alpha_i^{-1}(K_i,L_i)^t\,=\,\alpha_j^{-1}(K_j,L_j)^t\,.
\end{equation}
If we define $f_1$ and $f_2$ on $U_i$ by
$$
(f_1,f_2)^t=\alpha_i^{-1}(K_i,L_i)^t\,,
$$
then using \eqref{patch}, we see that  $f_1$ and $f_2$ are well defined all over $D$ and they are two linearly independent solutions to $y''+gy=0$ on all of $D$ as they are linearly independent over $U_i$.
\end{proof}

\section{The correspondence}
In this section, we will prove that every $\rho-$equivariant function arises from a vector-valued modular form as in \propref{direct}.
We start with the following property of
the slash operator  known as Bol's identity.

\begin{prop} Let $r$ be a nonnegative integer, $F(z)$ a complex function and $\gamma\in\SC$, then
$$
(F|_{-r}\gamma)^{(r+1)}(z)\,=\,F^{(r+1)}|_{r+2}\gamma(z)\,.
$$
\end{prop}

As a consequence, we have

\begin{cor}\label{cor4.2}
Let $r$ be a nonnegative integer, $g$ a $\Gamma-$automorphic form of weight $2(r+1)$ and $D$ a domain in $\fH$ that is stable under the action of $\Gamma$. Denote by $V_r(D)$  the solution space on $D$ to the differential equation
$$
f^{(r+1)}+gf=0\,.
$$
Then for all $\gamma\in\Gamma$,
$$
f\in V_r(D) \ \mbox{ if and only if } f|_{-r}\gamma\in V_r(D)\,.
$$
\end{cor}
\begin{cor}
The operator $|_{-r}$ provides a representation $\rho_r$ of $\Gamma$ in $\mbox{GL}(V_r)$. Moreover, if $f_1$, $f_2$,\ldots,$f_{r+1}$ form a basis of $V_r$ (if the basis exists), then
$$F=(f_1,f_2,\ldots,f_{r+1})^t
$$
behaves as a $\rho_r-$VMF of weight $-r$ for $\Gamma$.
\end{cor}

We know state the main result of this paper. Recall from \propref{direct} that if $F=(f_1,f_2)^t$ is a $\rho-$VMF, then $h_F=f_1/f_2$ is a $\rho-$equivariant function.

\begin{thm}
The map
$$V_{\rho}(\Gamma)_{-1}\rightarrow \e$$
$$F\mapsto h_F$$
is surjective.
\end{thm}
\begin{proof}
Suppose  that $h$ is a $\rho-$equivariant function for $\Gamma$. According \propref{s-autom},  its Schwarz derivative $S(h)$ is an automorphic form of weight 4 for $\Gamma$. Let $g=\frac{1}{2}S(h)$ and $D$ the complement in $\fH$ of the set of poles of $g$.
Then $D$ is a domain that is stable under $\Gamma$ since $g$ is an automorphic form for $\Gamma$.

Using the same notation as in the previous section, we have, for $r=1$, $S(f_1/f_2)=S(h)$ where $\{f_1,f_2\}$ are two linearly independent
solutions in $V(D)$ provided by \thmref{thm3.3}.
Hence, by \propref{schwarz}
$$
\frac{f_1}{f_2}\,=\,\alpha\cdot h\ ,\quad \alpha\in\GC\,.
$$
Also, using \corref{cor4.2} with $r=1$, we deduce that $F_1=(f_1,f_2)^t$ is a $\rho_1-$VMF of weight $-$1 for $\Gamma$. Therefore,
$$\alpha^{-1}\rho_1\alpha=\rho\,.$$
Hence $F=\alpha^{-1} F_1$ is a $\rho-$VMF of weight $-$1 for $\Gamma$ with $h_F=h$ on $D$. Since $g$ has only double poles, then by looking at the form
of the solutions near a singular point, and using the fact that $f_1$ and $f_2$ are holomorphic and thus single-valued, we see that $f_1$ and $f_2$
can be extended to meromorphic functions on $\fH$.
\end{proof}
\begin{remark}
If $f$ is an automorphic form of weight $k+1$ for $\G$, then $(f_1,f_2)\longrightarrow (ff_1,ff_2)$ yields an isomorphism between $V_{\rho}(\Gamma)_{-1}$ and $V_{\rho}(\Gamma)_{k}$, and therefore, the above surjection in the theorem extends to $V_{\rho}(\Gamma)_{k}$ whenever
it is nontrivial.
\end{remark}

\section{Examples}
In this section, we shall construct examples of $\rho-$VMF's and  of $\rho-$equivariant functions when $\rho$ is the monodromy representation
of a differential equation.

Let U be a domain in $\mathbb C$ such that $\mathbb C\setminus U$ contains at least two points. The universal covering of $U$ is then $\fH$
as it cannot be ${\mathbb P}_1(\mathbb C)$ because $U$ is noncompact and it cannot be $\mathbb C$ because of Picard's theorem.

Let $\pi: \fH\longrightarrow U$ be the covering map. We consider the differential equation on $U$
\begin{equation}\label{diff1}
y''+Py'+Qy=0
\end{equation}
where $P$ and $Q$ are two holomorphic functions on $U$. This differential equation has a lift to $\fH$
\begin{equation}\label{diff2}
y''+\pi^*Py'+\pi^*Qy=0\,.
\end{equation}
Let $V$ be the solution space to \eqref{diff2} which is a 2-dimensional vector space since $\fH$ is simply connected. Let $\gamma$ be a
covering transformation in Deck$(\fH/U)$ which is isomorphic to the fundamental group $\pi_1(U)$  and let $f\in V$.
Then $\gamma^*f=f\circ \gamma^{-1}$ is also a solution in $V$. This defines the monodromy representation of $\pi_1(U)$:
$$
\rho:\pi_1(U)\longrightarrow \mbox{GL}(V)\,.
$$
If $f_1$ and $f_2$ are two linearly independent solutions in $V$, we set $F=(f_1,f_2)^t$. Then we have
$$
F\circ\gamma\,=\,\rho(\gamma)F\,.
$$
Therefore, the quotient $f_1/f_2$ is a $\rho-$equivariant function on $\fH$ for the group $\pi_1(U)$ which is a torsion-free
 discrete group.

\end{document}